\documentclass[reqno,12pt]{amsart}

\usepackage{amsmath}
\usepackage{amsfonts}
\usepackage{amssymb}
\usepackage{eufrak}
\usepackage{amscd}
\usepackage{amsthm}
\usepackage{amstext}
\usepackage[all]{xy}

\newcommand{\Ker}{\operatorname{Ker}}

\newcommand{\ev}{\operatorname{ev}}

   \theoremstyle{plain}
   \newtheorem{thm}{Theorem}
   
   \newtheorem{lem}[thm]{Lemma}
   \newtheorem{cor}[thm]{Corollary}
   \theoremstyle{definition}
   
   \newtheorem{defn}[thm]{Definition}
   
   \theoremstyle{remark}

\author{V. Manuilov}

\date{}

\address{Moscow State University,
Leninskie Gory 1, Moscow, 
119991, Russia}

\email{manuilov@mech.math.msu.su}

\title{A more symmetric picture for Kasparov's $KK$-bifunctor}

\begin{document}

\maketitle

\begin{abstract}
For $C^*$-algebras $A$ and $B$, we generalize the notion of a quasihomomorphism from $A$ to $B$, due to Cuntz, by considering 
quasihomomorphisms from some $C^*$-algebra $C$ to $B$ such that $C$ surjects onto $A$, and the two maps forming a quasihomomorphism agree on the kernel of this surjection.
Under an additional assumption, the group of homotopy classes of such generalized quasihomomorphisms coincides with $KK(A,B)$. This makes the definition of Kasparov's bifunctor slightly more symmetric and gives more flexibility for constructing elements of $KK$-groups. These generalized quasihomomorphisms can be viewed as pairs of maps directly from $A$ (instead of various $C$'s), but these maps need not be $*$-homomorphisms.

\end{abstract}

\section*{Introduction}

Effectiveness of Kasparov's $KK$-bifunctor is caused by the fact that it unifies covariant and contravariant $K$-theory and generalizes morphisms of $C^*$-algebras. For $C^*$-algebras $A$ and $B$, any $*$-homomorphism $\varphi:A\to B$ gives rise to an element of $KK(A,B)$, but there may be too few $*$-homomorphisms to make a computable bifunctor. On this way, the notion of quasihomomorphism was coined by J. Cuntz \cite{Cuntz-qA}. Let 
$$
\begin{xymatrix}{
0\ar[r]& B\ar[r]& E\ar[r]^-{q}& Q\ar[r] &0
}\end{xymatrix}
$$ 
be a short exact sequence of $C^*$-algebras such that $B$ is an essential ideal in $E$. Then a quasihomomorphism from $A$ to $B$ is a pair of $*$-homomorphisms $\varphi_+,\varphi_-:A\to E$ such that $q\circ\varphi_+=q\circ\varphi_-$, and $KK(A,B)$ is the group of homotopy classes of such quasihomomorphisms when $B$ is stable. This notion works perfectly, in particular, it helps to simplify the Kasparov product, which is the ``composition'' of quasihomomorphisms, but aesthetically it lacks symmetry: when we generalize the notion of $*$-homomorphism, we replace the target ($B$), but don't change the sourse ($A$). Our aim is to diminish this imbalance. We replace not only $B$, but $A$ as well. Let 
$$
\begin{xymatrix}{
0\ar[r]& J\ar[r]^-{\iota}& C\ar[r]^-{p}& A\ar[r] &0
}\end{xymatrix}
$$ 
be a short exact sequence of $C^*$-algebras. We want to construct $KK(A,B)$ from pairs of maps $\varphi_+,\varphi_-:C\to E$ such that
\begin{enumerate}
\item
$q\circ\varphi_+=q\circ\varphi_-$;
\item
$\varphi_+\circ\iota=\varphi_-\circ\iota$.
\end{enumerate}

The second condition is here symmetric to the first one. We call maps satisfying (2) pseudohomomorphisms. An advantage here is that pseudohomomorphisms can be considered as (pairs of) maps directly from $A$ to $E$ (as opposed to maps from $C$ to $E$), although these maps need not to be $*$-homomorphisms. 

Besides the aesthetics, this generalization gives more flexibility and thus may allow to find more elements of the $KK$-groups from geometric constructions. 

In order to obtain $KK(A,B)$ we need to impose an additional requirement: the map $\varphi_+\circ\iota=\varphi_-\circ\iota$ should be continuous with respect to the strict topologies on $J$ and $E$. If we exclude this requirement then we get a (possibly) different bifunctor $KM(A,B)$, which contains $KK(A,B)$ as a direct summand.


\section{Some trivialities on amalgamated free products}

Let $A\ast A$ be the free product of two copies of $A$, and let $qA\subset A\ast A$ be the kernel of the canonical map $m_A:A\ast A\to A$ given by multiplication. 

Let $p:C\to A$ be a surjective $*$-homomorphism, $J=\Ker p$, and let $C\ast_J C$ be the amalgamated free product of two copies of $C$ over $J$. The canonical surjection $m_C:C\ast C\to C$ factorizes through $m:C\ast_J C\to C$. The map $p\ast p:C\ast C\to A\ast A$ also factorizes through $\bar p:C\ast_J C\to A\ast A$. Restricted to $\Ker m$, the map $\bar p$ gives a $*$-homomorphism $p':\Ker m\to qA$.

\begin{lem}\label{qA=}
The $*$-homomorphism $p':\Ker m\to qA$ is an isomorphism for any surjection $p:C\to A$.

\end{lem}
\begin{proof}
Consider the commuting diagram
\begin{equation}\label{diagram1}
\begin{xymatrix}{
0\ar[r]&\Ker m\ar[r]\ar[d]^-{p'}&C\ast_J C\ar[r]^-{m}\ar[d]^-{\bar p}&C\ar[r]\ar[d]^-{p}&0\\
0\ar[r]&qA\ar[r]&A\ast A\ar[r]^-{m_A}&A\ar[r]&0
}\end{xymatrix}
\end{equation}
with exact lines, where $\bar p$ and $p$ are surjective. Our aim is to show that $\Ker\bar p\cong J$ and that $m|_{\Ker\bar p}:\Ker\bar p\to J$ is an isomorphism. Then the Snake Lemma finishes the job.

Let $\iota^+_C,\iota^-_C:C\to C\ast C$ be the canonical $*$-homomorphisms to the first and the second copy of $C$. Slightly abusing the notation, we use the same $\iota^\pm_C$ to denote also the canonical $*$-homomorphisms from $C$ to the first and the second copy of $C$ in $C\ast_J C$. Define $i:J\to C\ast_J C$ by $i(j)=\iota^+_C(j)$, $j\in J\subset C$. Then $i(J)$ is an ideal in $C\ast_J C$. Indeed, recall that $C\ast_J C$ is the quotient of $C\ast C$ by the ideal generated by $\iota^+_C(j)-\iota^-_C(j)$, $j\in J$. As $C\ast_J C$ is generated, as a $C^*$-algebra, by $\iota^+_C(c)$ and $\iota^-_C(c)$, $c\in C$, so it suffices to check that $i(j)\iota^+_C(c)$ and $i(j)\iota^-_C(c)$ lie in $i(J)$ for any $c\in C$, and we obviously have $i(j)\iota^+_C(c)=\iota^+_C(jc)$ and $i(j)\iota^-_C(c)=\iota^-_C(j)\iota^-_C(c)=\iota^-_C(jc)=\iota^+_C(jc)$ in $C\ast_J C$. 
Note that $m(i(j))=j$ for any $j\in J$, hence $i$ is injective. As $\bar p(i(j))=0$ for any $j\in J$, $\bar p$ factorizes through a $*$-homomorphism $\hat p:C\ast_J C/i(J)\to A\ast A$ such that $\hat p(\iota^\pm_C(c))=\iota^\pm_A(p(c))$. It suffices to show that this $*$-homomorphism is bijective. 

To this end, consider the map 
$$
\begin{xymatrix}{
\hat\iota^+_C:C\ar[r]^-{\iota^+_C}& C\ast C\ar[r]&C\ast_J C/i(J).
}\end{xymatrix}
$$
It is easy to see that $\hat\iota^+_C(j)=0$ for any $j\in J$, hence $\hat\iota^+_C$ factorizes through the $*$-homomorphism $\gamma^+:A\to C\ast_J C/i(J)$. Similarly, we obtain a $*$-homomorphism $\gamma^-:A\to C\ast_J C/i(J)$ from $\iota^-_C$. By the universal property of the free product, the maps $\gamma^+$ and $\gamma^-$ give rise to a $*$-homomorphism $A\ast A\to C\ast_J C/i(J)$, which is obviously inverse to $\hat p$. 

\end{proof}

Thus, we can view $qA$ as an ideal in $C\ast_J C$ for any $C$ that surjects onto $A$, and get an extension
\begin{equation}\label{exact1}
\begin{xymatrix}{
0\ar[r]&qA\ar[r]&C\ast_J C\ar[r]^-{m}&C\ar[r]&0.
}\end{xymatrix}
\end{equation}

We also have an extension
\begin{equation}\label{exact2}
\begin{xymatrix}{
0\ar[r]&J\ar[r]&C\ast_J C\ar[r]^-{\bar p}&A\ast A\ar[r]&0
}\end{xymatrix}
\end{equation}

By Proposition 3.6 in \cite{Pedersen}, the two extensions (\ref{exact1}) and (\ref{exact2}) give a new extension
\begin{equation}\label{exact3}
\begin{xymatrix}{
0\ar[r]&qA\cap J\ar[r]&C\ast_J C\ar[r]&(A\ast A)\oplus_X C\ar[r]&0,
}\end{xymatrix}
\end{equation}
where $X=C\ast_J C/(qA+J)$. As $qA\cap J=0$ and $X=A$, we obtain an isomorphism $C\ast_J C\cong (A\ast A)\oplus_A C$.





\section{Some trivialities on strict topology}\label{strict}

Let $J$ be a 
$C^*$-algebra, $M(J)$ its multiplier algebra. The strict topology on $J$ (and on $M(J)$) is defined by the family of seminorms $p_j(x)=\|jx\|+\|xj\|$, $j\in J$ (here $x\in J$ or $x\in M(J)$).

Let $B$ be stable and $\sigma$-unital, and let $IB=B[0,1]$. 




Let $U_t$, $t\in(0,1]$, be a family of isometries in $M(B)$ such that 
\begin{itemize}
\item[(u1)]
$u_1=1$;
\item[(u2)]
the map $t\mapsto U_t$ is strictly continuous on $(0,1]$;
\item[(u3)]
$U_tU^*_t$ is strictly convergent to 0 as $t\to 0$.
\end{itemize}
For a $*$-homomorphism $\pi:J\to M(B)$, set 
$$
\pi_t(\cdot)=\left\lbrace\begin{array}{cl}U_t\pi(\cdot)U_t^*,&\mbox{if\ }t\in(0,1];\\
0,&\mbox{if\ }t=0.\end{array}\right.
$$ 
Denote by $\tilde{\pi}(\cdot)$ the map $t\mapsto\pi_t(\cdot)$. It is known that $\tilde{\pi}(j)\in M(IB)$ for any $j\in J$.

\begin{lem}\label{strict1}
Let $\pi:J\to M(B)$ be a $*$-homomorphism, continuous with respect to the strict topologies on $J$ and $M(B)$. Then $\tilde{\pi}:J\to M(IB)$ is continuous with respect to the strict topologies on $J$ and $M(IB)$. 

\end{lem}
\begin{proof}
As $J$ is dense in $M(J)$, $\pi$ extends to a map on $M(J)$. Let $(j_\lambda)_{\lambda\in\Lambda}$ be a net in $M(J)$ strictly convergent to 0. By assumption, $\pi(j_\lambda)$ is then strictly convergent to 0, so, by Banach--Steinhaus Theorem,
the set $\{\pi(j_\lambda):\lambda\in\Lambda\}$ is norm-bounded, i.e. there exists $C$ such that $\sup_{\lambda\in\Lambda}\|\pi(j_\lambda)\|<C$.

Let $b\in IB$, $b_t=\ev_t(b)\in B$. 
We claim that 
$$
\lim_{\Lambda}\|\tilde{\pi}(j_\lambda)b\|=\lim_{\Lambda}\sup_{t\in[0,1]}\|\pi_t(j_\lambda)b_t\|=0.
$$ 

Assume the contrary: there exists $\delta>0$ such that for any $\mu\in\Lambda$ there is $\lambda_\mu\geq\mu$ and $t_\mu\in[0,1]$ such that $\|\pi_{t_\mu}(j_{\lambda_\mu})b_{t_\mu}\|>\delta$. Let $t_0$ be an accumulation point for the net $\{t_\mu\}_{\mu\in\Lambda}$, i.e. for every $\varepsilon>0$ and for every $\nu\in\Lambda$ there exists $\mu\geq \nu$ such that $|t_\mu-t_0|<\varepsilon$. Consider the two cases: (a) $t_0>0$ and (b) $t_0=0$. 

Case (a): In this case we may assume that $t_\lambda>0$ for any $\lambda\in\Lambda$. Then
\begin{equation}\label{1}
\delta<\|\pi_{t_\mu}(j_{\lambda_\mu})b_{t_\mu}\|=\|U_{t_\mu}\pi(j_{\lambda_\mu})U^*_{t_\mu}b_{t_\mu}\|\leq \|\pi(j_{\lambda_\mu})U^*_{t_\mu}b_{t_\mu}\|.
\end{equation}
As $t\mapsto U^*_t$ is strictly continuous on $(0,1]$, the product $t\mapsto h_t=U^*_tb_t$ is norm-continuous at $t_0$, so for any $\nu\in\Lambda$ there exists $\mu\geq\nu$ such that $\|h_{t_\mu}-h_{t_0}\|<\delta/2C$.
Then, for this $\mu$, 
\begin{equation}\label{2}
\|\pi(j_{\lambda_\mu})U^*_{t_\mu}b_{t_\mu}\|\leq 
\|\pi(j_{\lambda_\mu})\|\ \|h_{t_\mu}-h_{t_0}\|+\|\pi(j_{\lambda_\mu})h_{t_0}\|
< \delta/2+\|\pi(j_{\lambda_\mu})h_{t_0}\|.
\end{equation}
It follows from (\ref{1}) and (\ref{2}) that for any $\nu\in\Lambda$ there exists $\mu\geq\nu$ such that $\|\pi(j_{\lambda_\mu})h_{t_0}\|>\delta/2$. 
This contradicts strict continuity of $\pi$.

Case (b): Here we have
\begin{eqnarray*}
\delta^2&<&\|\pi_{t_\mu}b_{t_\mu}\|^2=\|b_{t_\mu}^*\pi_{t_\mu}(j_{\lambda_\mu}^*j_{\lambda_\mu})b_{t_\mu}\|=\|b_{t_\mu}^*U_{t_\mu}\pi(j_{\lambda_\mu}^*j_{\lambda_\mu})U^*_{t_\mu}b_{t_\mu}\| \\
&\leq& \|\pi(j_{\lambda_\mu}^*j_{\lambda_\mu})b_{t_\mu}^*U_{t_\mu}U^*_{t_\mu}b_{t_\mu}\|\leq C^2\|b\|\ \|U_{t_\mu}U^*_{t_\mu}b_{t_\mu}\|,
\end{eqnarray*}
i.e. for any $\mu\in\Lambda$ there exists $t_\mu\in[0,1]$ such that $\|U_{t_\mu}U_{t_\mu}^*b_{t_\mu}\|>\frac{\delta^2}{C\|b\|}$.
This contradicts strict continuity of the map $t\mapsto U_tU^*_t$.

Similarly, one can show that $\lim_{\Lambda}\|b\tilde{\pi}(j_\lambda)\|=0$, therefore, $\tilde{\pi}$ is continuous with respect to the strict topologies.

\end{proof}

\begin{lem}\label{strict2}
Let $q:J'\to J$ be a surjective $*$-homomorphism, and let $\pi:J\to M(B)$ be a $*$-homomorphism, continuous with respect to the strict topologies on $J$ and $M(B)$. Then $\pi\circ q$ is continuous with respect to the strict topologies on $J'$ and $M(B)$.

\end{lem}
\begin{proof}
Surjectivity of $q$ implies that $q$ is continuous with respect to the strict topologies on $J'$ and on $J$. Indeed, let $V_j=\{x\in J:p_j(x)<1\}\subset J$, $j\in J$. Let $j'\in J'$, $q(j')=j$. Set $W_{j'}=\{x'\in J':p_{j'}(x')<1\}\subset J'$. Then $W_{j'}$ is open and $q(W_{j'})\subset V_j$.

\end{proof}

\section{A technical lemma}

Recall that $M(B)$ can be considered as the algebra of adjointable bounded operators on the Hilbert $C^*$-module $B$ over itself when $B$ is stable. A $*$-homomorphism $\pi:J\to M(B)$ is quasi-unital \cite{Thomsen-Duke} if there exists a projection $e\in M(B)$ such that $\overline{\pi(J)B}=eB$. Lemma 2.14 of \cite{Thomsen-Duke} shows that $\pi$ is quasi-unital iff it is strictly continuous, and that in this case it admits a unique extension to a strictly continuous $*$-homomorphism $\bar\pi:M(J)\to M(B)$.

\begin{lem}\label{technical}
Let $J$ be an ideal in a separable $C^*$-algebra $D$, let $\pi:D\to M(B)$ be a $*$-homomorphism such that $\pi|_J$ is quasi-unital, and let $e\in M(B)$ be the corresponding projection. Then 
\begin{itemize}
\item[(1)]
$[\pi(d),e]=0$ for any $d\in D$
\item[(2)]
if $dJ=0$ then $\pi(d)e=0$.
\end{itemize}
\end{lem}
\begin{proof}
Set $J^\perp=\{d\in D:dJ=0\}$. Then $J^\perp$ is an ideal in $D$, and $J\cap J^\perp=0$. It is easy to see that $J\oplus J^\perp$ is an essential ideal in $D$, hence 
$$
J\oplus J^\perp\subset D\subset M(J\oplus J^\perp)=M(J)\oplus M(J^\perp).
$$ 
Let $q_1$ and $q_2$ be the projections of $M(J)\oplus M(J^\perp)$ onto the first and the second summands respectively, and let $D'=p_2(D)$. Then $D\subset M(J)\oplus D'$. Let us construct a $*$-homomorphism $\bar\pi:M(J)\oplus D'\to M(B)$ that extends $\pi$.

The extension of $\pi|_J$ to $M(J)$ exists due to strict continuity of $\pi|_J$, so it remains to define $\bar\pi$ on $D'$. Let $a'\in D'$. Then there is $a\in D$ such that $q_2(a)=a'$. Set
$$
\bar\pi(a')=\pi(a)-\bar\pi(q_1(a)).
$$ 
If $q_2(b)=a$ then $q_2(a-b)=0$, hence $a-b=q_1(a-b)\in D\cap M(J)$, so 
$$
\pi(a-b)=\pi(q_1(a-b))=\bar\pi(q_1(a))-\bar\pi(q_1(b)). 
$$
Thus, $\bar\pi$ is well-defined.

Let us check multiplicativity of $\bar\pi|_{D'}$. Let $a,b\in D$, $q_2(a)=a'$, $q_2(b)=b'$, then
$$
\bar\pi(a'b')-\bar\pi(a')\bar\pi(b')=\pi(a)\bar\pi(q_1(b))+\bar\pi(q_1(a))\pi(b)-2\bar\pi(q_1(ab)).
$$
Let $\{u_t\}_{t\in[0,\infty)}$ be an approximate unit in $J$, quasicentral in $D$. Then 
$$
q_1(a)=s\mbox{-}\lim_{t\to\infty}q_1(a)u_t=s\mbox{-}\lim_{t\to\infty}q_1(au_t), 
$$
where $s\mbox{-}\lim$ denotes the strict limit, so
$$
\bar\pi(q_1(a))\pi(b)-\bar\pi(q_1(ab))=s\mbox{-}\lim_{t\to\infty}\pi(q_1(au_t))\pi(b)-\pi(q_1(abu_t))=
s\mbox{-}\lim_{t\to\infty}\pi(q_1(au_tb-abu_t))=0.
$$
Similarly, $\pi(a)\bar\pi(q_1(b))-\bar\pi(q_1(ab))=0$, hence $\bar\pi|_{D'}$ is multiplicative. Finally, if $m\in M(J)$, $a'\in D'$ then 
$$
\bar\pi(m)\bar\pi(a')=s\mbox{-}\lim_{t\to\infty}\pi(mu_t)(\pi(a)-\bar\pi(q_1(a)))=s\mbox{-}\lim_{t\to\infty}\pi(mu_ta-mu_tq_1(a))=0.
$$
Thus, $\bar\pi:M(J)\oplus D'\to M(B)$ is a $*$-homomorphism. If $d\in D$ then 
$$
\bar\pi(d)=\bar\pi(q_1(d))+\bar\pi(q_2(d))=\bar\pi(q_1(d))+\pi(d)-\bar\pi(q_1(d))=\pi(d),
$$
so $\bar\pi$ extends $\pi$.

Let $1\in M(J)$, $(1,0)\in M(J)\oplus D'$. Then $[(1,0),d]=0$ for any $d\in D$. As $\bar\pi(1,0)=e$, we have $[e,\bar\pi(d)]=0$. If $dJ=0$ then $d(1,0)=0$, hence $\pi(d)e=0$.

\end{proof}

\section{Generalized $KK$-cycles}

Let $A$ be a separable $C^*$-algebra. Consider all surjective $*$-homomorphisms $p:C\to A$, where $C$ is a separable $C^*$-algebra, with a partial order given by $(C,p)\leq (C',p')$ if there exists a surjective $*$-homomorphism $\lambda:C'\to C$ such that $p'=p\circ\lambda$. Denote the set of all such $C$ by $\mathcal E_A$. Abusing the notation, we shall write $C$ in place of $(C,p)$. The set $\mathcal E_A$ is directed. Indeed, if $p_1:C_1\to A$, $p_2:C_2\to A$ are surjections then the pull-back 
$$
C=\{(c_1,c_2):c_1\in C_1,c_2\in C_2,p_1(c_1)=p_2(c_2)\}
$$ 
obviously surjects onto $C_1$, $C_2$ and $A$, hence satisfies $C\geq C_1$, $C\geq C_2$. 

This directed set has a minimal and a maximal elements. The minimal element is $C=A$, and the maximal element was constructed in \cite{Cuntz2}, Section 2, under the name of {\it universal extension}.

Let $C\in\mathcal E_A$, and let $J=\Ker p$. Let $B$ be a stable $\sigma$-unital $C^*$-algebra, $M(B)$ its multiplier algebra.
\begin{defn}
A generalized $KK$-cycle from $A$ to $B$ is a pair $(\varphi_+,\varphi_-)$ of $*$-homomorphisms $\varphi_\pm:C\to M(B)$, where $C\in\mathcal E_A$, such that 
\begin{itemize}
\item[(1)] $\varphi_+(c)-\varphi_-(c)\in B$ for any $c\in C$;
\item[(2)] $\varphi_+|_J=\varphi_-|_J$.
\end{itemize}

A generalized $KK$-cycle is strict if $\varphi_+|_J:J\to M(B)$ is continuous with respect to the strict topologies on $J$ and on $M(B)$. 

Two (strict) generalized $KK$-cycles $(\varphi_+^0,\varphi_-^0)$ and $(\varphi_+^1,\varphi_-^1)$ from $A$ to $B$, with given $C\in\mathcal E_A$, are homotopic if there is a (strict) generalized $KK$-cycle $(\Phi_+,\Phi_-)$ from $A$ to $IB=C([0,1];B)$ (with the same $C$) such that the evaluation maps at 0 and at 1 give $(\varphi_+^0,\varphi_-^0)$ and $(\varphi_+^1,\varphi_-^1)$.

\end{defn}

Let $p:C\to A$ be a surjection, $J=\Ker p$, and let $KM(C,J;B)$ be the set of homotopy equivalence classes of generalized $KK$-cycles $(\varphi_+,\varphi_-)$ from $A$ to $B$ (with the given $C$). It has a natural structure of an abelian semigroup with the zero element $(0,0)$ due to stability of $B$, and it is easy to see that  $(\varphi,\varphi)$ is the trivial element: let $U_t$, $t\in(0,1]$ be the family of isometries as in Section \ref{strict}. Then the required homotopy is given by 
$$
\varphi_t(\cdot)=\left\lbrace\begin{array}{cl}U_t\varphi(\cdot)U_t^*,&\mbox{\ if\ }t\in(0,1];\\0,&\mbox{\ if\ }t=0.\end{array}\right.
$$ 
Thus, $KM(C,J;B)$ is an abelian group.
Set 
$$
KM(A,B)=\injlim_{C\in\mathcal E_A}KM(C,J;B)
$$ 
(it is easy to see that if $C\leq C'$ in $\mathcal E_A$ then the composition with the map $C'\to C$ gives a canonical map $KM(C,J;B)\to KM(C',J';B)$, where $J'$ is the kernel of the surjection $C'\to A$).

Similarly, let $KS(C,J;B)$ be the set of homotopy equivalence classes of strict generalized $KK$-cycles $(\varphi_+,\varphi_-)$ from $A$ to $B$ (with the given $C$). It is an abelian group by the same argument, taking into account Lemma \ref{strict1}. 
Set 
$$
KS(A,B)=\injlim_{C\in\mathcal E_A}KS(C,J;B)
$$ 
(it is easy to see that if $C\leq C'$ in $\mathcal E_A$ then, by Lemma \ref{strict2}, the composition with the map $C'\to C$ gives a map $KS(C,J;B)\to KS(C',J';B)$, where $J'$ is the kernel of the surjection $C'\to A$). Forgetting about strict continuity, we get a map $KS(A,B)\to KM(A,B)$.

As $A\in\mathcal E_A$, there are canonical maps $i_M:KK(A,B)\to KM(A,B)$ and $i_S:KK(A,B)\to KS(A,B)$. 

Now let $C\in\mathcal E_A$.
Note that, by the universal property of the free product of $C^*$-algebras, any two $*$-homomorphisms $\varphi_\pm:C\to M(B)$ with $\varphi_+(c)-\varphi_-(c)\in B$, $c\in C$, define a $*$-homomorphism $\varphi_+\ast\varphi_-:C\ast C\to M(B)$ such that $\varphi_+\ast\varphi_-(qC)\subset B$. If the pair $(\varphi_+,\varphi_-)$ is a (strict) generalized $KK$-cycle then $\varphi_+|_J=\varphi_-|_J$, hence $\varphi_+\ast\varphi_-$ factorizes through $C\ast_J C$. By Lemma \ref{qA=}, $qA$ is an ideal in $C\ast_J C$, and restricting this map onto $qA\subset C\ast_J C$, we obtain a $*$-homomorphism $q(\varphi_+,\varphi_-):qA\to B$. As there is a canonical isomorphism $KK(A,B)\cong [qA,B]$, where $[X,Y]$ denotes the set of homotopy classes of $*$-homomorphisms from $X$ to $Y$ (which is an abelian group when $X=qA$ and $Y$ is stable), \cite{Cuntz-qA}, we obtain homomorphisms $j_M:KM(A,B)\to KK(A,B)$ and $j_S:KS(A,B)\to KK(A,B)$.

\section{Generalized $KK$-cycles as $KK$-bifunctor}

\begin{thm}
Let $A$ be separable and $B$ $\sigma$-unital and stable. Then the groups $KS(A,B)$ and $KK(A,B)$ are canonically isomorphic, and $KM(A,B)$ contains $KK(A,B)$ as a direct summand.

\end{thm}

{\it Proof} follows from the next two Lemmas.

\begin{lem}\label{mono}
$j_M\circ i_M$ and $j_S\circ i_S$ equal the identity map on $KK(A,B)$. 

\end{lem}
\begin{proof}
Let $(\varphi_+,\varphi_-)$ represent an element of $KK(A,B)$. Then $j_M\circ i_M([(\varphi_+,\varphi_-)])=j_S\circ i_S([(\varphi_+,\varphi_-)])=q(\varphi_+,\varphi_-)$, and the claim follows from the identification $[qA,B]\cong KK(A,B)$.

\end{proof}

In general, we cannot prove that $i_M\circ j_M$ is the identity map on $KM(A,B)$, but

\begin{lem}\label{epi}
$i_S\circ j_S$ equals the identity map on $KS(A,B)$.

\end{lem}
\begin{proof}
Let $C\in\mathcal E_A$, and let $(\varphi_+,\varphi_-)$ be a strict generalized $KK$-cycle from $A$ to $B$. 
We claim that there exists a strict generalized $KK$-cycle $(\psi_+,\psi_-)$ from $A$ to $B$ (with the same $C$) such that
\begin{itemize}
\item[(1)]
$(\psi_+,\psi_-)$ is homotopic to $(\varphi_+,\varphi_-)$;
\item[(2)]
$\psi_\pm$ factorize through $A$, i.e. there exists a $KK$-cycle $(\mu_+,\mu_-)$ from $A$ to $B$ such that $\psi_\pm=\mu_\pm\circ p$.
\end{itemize}
If the claim holds true then $j_S(\varphi_+,\varphi_-)=j_S(\psi_+,\psi_-)=(\mu_+,\mu_-)$ and $i_S(\mu_+,\mu_-)=(\psi_+,\psi_-)$, hence $i_S\circ j_S(\varphi_+,\varphi_-)=(\psi_+,\psi_-)$, and we are done, so let us prove the claim.  

Let $\varphi_+\ast\varphi_-:C\ast_J C\to M(B)$ be the free product of $\varphi_+$ and $\varphi_-$. As $\varphi_+\ast\varphi_-$ is strictly continuous on $J\subset C\ast_J C$, by Lemma \ref{technical}, there is a projection $e\in M(B)$ such that 
\begin{equation}\label{x0}
\varphi_\pm(J)\subset eM(B)e,
\end{equation} 
\begin{equation}\label{x1}
[\varphi_+\ast\varphi_-(d),e]=0 \mbox{\ for\ any\ } d\in C\ast_J C,
\end{equation}
and 
\begin{equation}\label{x2}
\varphi_+\ast\varphi_-(d)\in (1-e)M(B)(1-e) \mbox{\ for\ any\ } d\in C\ast_J C \mbox{\ such\ that\ }dJ=0.
\end{equation}
It follows from (\ref{x1}) that 
\begin{equation}\label{x3}
[\varphi_\pm(c),e]=0 
\end{equation}
for any $c\in C$. As $J$ is an ideal in $C$,
$$
\varphi_+(c)\varphi_+(j)=\varphi_+(cj)=\varphi_-(cj)=\varphi_-(c)\varphi_+(j),
$$ 
hence 
\begin{equation}\label{x4}
(\varphi_+(c)-\varphi_-(c))\varphi_\pm(j)=0
\end{equation}
for any $c\in C$ and any $j\in J$. Then, it follows from  
(\ref{x2}) that $\varphi_+(c)-\varphi_-(c)\in(1-e)M(B)(1-e)$ for any $c\in C$.

Take two copies $B\oplus B$ of the Hilbert $C^*$-module $B$ over itself, and replace the strict generalized $KK$-cycle $(\varphi_+,\varphi_-)$ by $(0\oplus\varphi_+,0\oplus\varphi_-)$, where $0\oplus\varphi_\pm:C\to M_2(M(B))$ is given by $0\oplus\varphi_\pm(c)=\left(\begin{matrix}0&0\\0&\varphi_\pm(c)\end{matrix}\right)$, $c\in C$.

Decompose $B\oplus B$ as $B\oplus eB\oplus(1-e)B$. Then elements of $M_2(M(B))$ can be written as $3{\times}3$ matrices with respect to this decomposition. Then $0\oplus\varphi_\pm(c)=\left(\begin{array}{cc}0&\begin{array}{cc}0&0\end{array}\\\begin{array}{c}0\\0\end{array}& \varphi_\pm(c)\end{array}\right)$.
By Kasparov Stabilization Theorem, the Hilbert $C^*$-module $B\oplus eB$ is isomorphic to $B$, hence there exists a family of isometries $U_t$, $t\in(0,1]$, on $B\oplus eB$ such that (u1)-(u3) holds. Then $V_t=\left(\begin{array}{cc}U_t&\begin{array}{c}0\\0\end{array}\\\begin{array}{cc}0&0\end{array}&1\end{array}\right)$ is also a family of isometries with
$\lim_{t\to 0}V_t=\left(\begin{array}{cc}\text{\Large 0}&\begin{array}{c}0\\0\end{array}\\\begin{array}{cc}0&0\end{array}&1\end{array}\right)=P$. 

Set 
$$
\psi_{\pm,t}(\cdot)=\left\lbrace\begin{array}{cl}V_t(0\oplus\varphi_\pm(\cdot))V_t^*,&\mbox{\ for\ }t\in(0,1];\\
P(0\oplus\varphi_\pm(\cdot))P,&\mbox{\ for\ }t=0.\end{array}\right.
$$ 
When $t\in(0,1]$, $V_t$ is an isometry, hence $\varphi_{\pm,t}$ is a $*$-homomorphism. When $t=0$, it is a $*$-isomorphism too, due to (\ref{x3}). It follows from (\ref{x4}) that $\psi_{+,t}(c)-\psi_{-,t}(c)$ does not depend on $t$, hence is norm-continuous as a map from $[0,1]$ to $M_2(B)$ for any $c\in C$. Lemma \ref{strict1} implies that $(\psi_{+,t},\psi_{-,t})$ is a homotopy connecting $(0\oplus\varphi_+,0\oplus\varphi_-)$ with $(\psi_{+,0},\psi_{-,0})$. Finally, note that if $j\in J$ then (\ref{x0}) implies that $\psi_{\pm,0}(j)=0$, therefore, $\psi_{\pm,0}$ factorize through $A$, i.e. there exist $*$-homomorphisms $\mu_\pm:A\to M(B)$ such that $\psi_{\pm,0}=\mu_\pm\circ p$. As $(\psi_{+,0},\psi_{-,0})$ is a generalized $KK$-cycle, $(\mu_+,\mu_-)$ is a $KK$-cycle.

\end{proof}

\begin{cor}
Let $B=\mathbb K$. Then $i_M\circ j_M$ is the identity map on $KM(A,\mathbb K)$, hence $KM(A,\mathbb K)\cong KK(A,\mathbb K)$.

\end{cor}
\begin{proof}
Any Hilbert $C^*$-submodule over $\mathbb K$ is complementable \cite{Magajna}, hence any $*$-homomorphism to $M(\mathbb K)$ is quasi-unital, hence continuous with respect to the strict topologies.

\end{proof}

\section{Pseudohomomorphisms as maps}

Let $(\mu_+,\mu_-)$ be a pair of maps $\mu_\pm:A\to E$, not necessarily additive or multiplicative, but homogeneous and involutive. We denote by $F=F(\mu_+,\mu_-)=C^*(\mu_+(A),\mu_-(A))\subset E$ the $C^*$-algebra generated by all $\mu_\pm(a)$, $a\in A$. Let $I=I(\mu_+,\mu_-)\subset F$ be the ideal, in $F$, generated by $\mu_+(a)-\mu_-(a)$, $a\in A$, and let $J=J(\mu_+,\mu_-)\subset F$ be the ideal, in $F$, generated by $M_\pm(a,b)$, $a,b\in A$, where $M_\pm(a,b)$ is either $\mu_\pm(a+b)-\mu_\pm(a)-\mu_\pm(b)$ or $\mu_\pm(ab)-\mu_\pm(a)\mu_\pm(b)$.

\begin{defn}\label{Def1}
A pair $(\mu_+,\mu_-)$ of continuous homogeneous involutive maps $\mu_\pm:A\to E$ {\it has the same deficiency} from being  a $*$-homomorphism if 
\begin{equation}\label{a2}
I\cap J=0. 
\end{equation}

\end{defn}

It is easy to see that this is equivalent (\cite{ME}) to 
$$
M_+(a,b)=M_-(a,b)\quad\mbox{and}\quad
M_+(a,b)\mu_+(c)=M_-(a,b)\mu_-(c)
$$
for any $a,b,c\in A$.


\begin{defn}
Let $p:C\to A$ be a surjection. 
A pair $(\varphi_+,\varphi_-)$ of $*$-homomorphisms $\varphi_\pm:C\to E$ is a {\it pseudohomomorphism} from $A$ to $E$ if $\varphi_+|_J=\varphi_-|_J$, where $J=\Ker p$.

\end{defn}

Let $C\in\mathcal E_A$, $(\psi_+,\psi_-)$ a pseudohomomorphism from $A$ to $B$. 
Let $s:A\to C$ be a $\mathbb C$-homogeneous $*$-respecting continuous map such that $p(s(a))=a$ for any $a\in A$, which exists by \cite{Bartle-Graves}. Set $\mu_\pm=\psi_\pm\circ s$. 
\begin{lem}
The pair $(\mu_+,\mu_-)$ has the same deficiency from being a $*$-homomorphism.

\end{lem}
\begin{proof}
Let us check (\ref{a2}). Note that 
\begin{eqnarray*}
(\mu_\pm(ab)-\mu_\pm(a)\mu_\pm(b))\mu_\pm(c)
&=&(\psi_\pm(s(ab))-\psi_\pm(s(a))\psi_\pm(s(b)))\psi_\pm(s(c))\\
&=&(\psi_\pm(s(ab)-s(a)s(b)))\psi_\pm(s(c))\\
&=& \psi_\pm((s(ab)-s(a)s(b))s(c)),
\end{eqnarray*}
and that $s(ab)-s(a)s(b)\in J$, so $(s(ab)-s(a)s(b))s(c)\in J$, and as $\psi_+$ and $\psi_-$ agree on $J$, so we are done. 

\end{proof}

Note that this construction depends on a choice of the map $s$, but as any two different maps $s,s':A\to C$ satisfying the above assumptions can be connected by the linear homotopy, so the resulting pairs of maps are homotopic in any reasonable sense.

Now, let $(\mu_+,\mu_-)$ be a pair of maps, $\mu_\pm:A\to E$, with the same deficiency from being a $*$-homomorphism. Let $q:F\to F/J$ and $r:F\to F/I$ be the quotient maps. Obviously, $q\circ\mu_\pm$ are $*$-homomorphisms from $A$ to $F/J$.  

Set
$$
C_\mu=\{(a,f_+,f_-):a\in A,f_+,f_-\in F, \mu_\pm(a)=q(f_\pm),r(f_+)=r(f_-)\}.
$$
This is a $C^*$-algebra that surjects onto $A$, $p(a,f_+,f_-)=a$. It has also two surjections, $p_+(a,f_+,f_-)=f_+$, $p_-(a,f_+,f_-)=f_-$, onto $F$. 

Note that 
$$
J_\mu=\Ker p=\{(0,f_+,f_-):f_\pm\in J,f_+-f_-\in I\}=\{(0,j,j):j\in J\},
$$ 
as $I\cap J=0$.
Then $p_+|_{I_\mu}=p_-|_{I_\mu}$. Let $\iota:F\to E$ denote the inclusion, and set $\varphi_\pm=\iota\circ p_\pm:C_\mu\to E$. Then the pair $(\varphi_+,\varphi_-)$ is a pseudohomomorphism from $A$ to $E$.

\end{document}